\documentclass[a4paper,12pt,oneside]{amsart}

\usepackage[utf8]{inputenc}
\usepackage[T1]{fontenc}

\usepackage[
  a4paper,
  left=3cm,
  right=3cm,
  top=4cm,
  bottom=4cm,
  twoside=false 
]{geometry}

\usepackage{amsmath,amsthm,verbatim,amssymb,amsfonts,amscd,graphicx,mathrsfs,array,float,mathtools,tikz-cd,enumitem}
\usepackage{tabularx}
\usepackage{diagbox}

\numberwithin{equation}{section}

\newtheorem{theorem}{Theorem}[section]
\newtheorem{corollary}[theorem]{Corollary}
\newtheorem{lemma}[theorem]{Lemma}
\newtheorem{proposition}[theorem]{Proposition}

\theoremstyle{definition}
\newtheorem{definition}[theorem]{Definition}

\theoremstyle{remark}
\newtheorem{remark}[theorem]{Remark}
\newtheorem{example}[theorem]{Example}
\newtheorem{notation}[theorem]{Notation}

\usepackage{thmtools}
\usepackage{thm-restate}
\usepackage{hyperref}
\usepackage{cleveref}

\newcommand{\Pic}{\mathrm{Pic}}

\newcommand{\PP}{\mathbb{P}}

\newcommand{\Ptwo}{\mathbb{P}^2}

\newcommand{\CC}{\mathbb{C}}
\newcommand{\GL}{\mathrm{GL}}
\newcommand{\Bir}{\mathrm{Bir}}
\newcommand{\Aut}{\mathrm{Aut}}


\begin{document}
\pagestyle{plain}

\title{Jordan bounds for Volume-Preserving Cremona Groups}

\author[J. Wang]{Jiahe Wang}
\address{UCLA Mathematics Department, Box 951555, Los Angeles, CA 90095-1555, USA
}
\email{jiahewang@math.ucla.edu}

\begin{abstract}
 We show that the Jordan constant for the volume-preserving plane Cremona group $\mathrm{Bir}(\Ptwo, \Delta)$ is $12$. We provide a Jordan bound of $144$ for the three-dimensional volume-preserving Cremona group $\mathrm{Bir}(\mathbb P^3,\Delta)$. We also provide a weak geometric Jordan bound of $2^{11} \cdot 3^2$ for $\Bir(\Ptwo)$. \end{abstract}

\maketitle
\tableofcontents

\section{Introduction}
\begin{definition}
A group $\mathcal{G}$ has the \emph{Jordan property}
if there exists a bound $J$ such that for every finite subgroup
$G \le \mathcal{G}$, there exists an abelian normal subgroup $A \trianglelefteq G$ with
\(
[G : A] \le J.
\)
The minimal such bound $J$ is called the \emph{Jordan constant}, denoted $J(\mathcal{G})$. Given a family of groups $\mathcal{F}$, we say that $\mathcal{F}$ has the \emph{Jordan property} if there exists a bound $J$ that applies to all finite subgroups of groups $\mathcal{G}$ in $\mathcal{F}$, and we denote the Jordan constant as $J(\mathcal{F})$.
\end{definition}
In 1878, Camille Jordan \cite{Jordan1878} showed that for any field $k$ of characteristic zero and any positive integer $n$, the linear automorphism group $\mathrm{GL}_n(k)$ satisfies the Jordan property. Collins \cite{Collins2007} found  that $J(\GL_n(k))=(n+1)!$ for $n \geq 71$. The modern theory of Jordan property for birational transformation groups has been developed systematically by many authors; see, for example, Popov's survey \cite{Popov2014JordanGroups} for a general overview. Serre \cite{serre2009minkowskistyleboundorderfinite} proved that the birational automorphism group of $\mathbb{P}^2$, $\mathrm{Bir}(\mathbb{P}^2)$, over characteristic zero possesses the Jordan property with Jordan constant
$
J(\mathrm{Bir}(\mathbb{P}^2)) \leq 2^{10} \cdot 3^4 \cdot 5^2 \cdot 7,
$
and Yasinsky later proved $J(\mathrm{Bir}(\mathbb{P}^2))=7200$ \cite{yasinsky2023jordanconstantcremonagroup}. Prokhorov and Shramov \cite{prok17} showed that $\mathrm{Bir}(\mathbb{P}^3)$ has the Jordan property and $J(\mathrm{Bir}(\mathbb{P}^3))\leq 107495424$. More generally, combining Birkar's boundedness of Fano varieties with their earlier work, Prokhorov and Shramov~\cite{prokhorov2014jordanpropertycremonagroups} proved that the family of groups \(\operatorname{Bir}(X)\), where \(X\) ranges over rationally connected varieties of fixed dimension \(n\) over fields of characteristic \(0\), satisfies the Jordan property; see also related results and refinements in \cite{BandmanZarhin2015JordanSurfaces,MengPerroniZhang2018JordanAut}. Many other groups are known to be Jordan. For example, the family of connected algebraic groups of a fixed dimension has the Jordan property \cite{luo2025jordanpropertyautomorphismgroups,MengZhang2018JordanNonlinear}. As a consequence, $\mathrm{Aut}(X)$ satisfies the Jordan property for any projective variety $X$. In the singular setting, for each fixed dimension \(n\), the family of groups \(\pi_1(\operatorname{Link}(X,x))\), where \((X,x)\) ranges over \(n\)-dimensional klt singularities, satisfies the Jordan property~\cite{BraunFilipazziMoragaSvaldi2022}.

The group we are interested in is the group of volume-preserving birational self-maps of $\mathbb{P}^2$ over the complex numbers:
\[
\mathrm{Bir}(\mathbb{P}^2,\Delta) = \{ f: \Ptwo \dashrightarrow \Ptwo \text{ birational map}, f^*\Omega = \lambda \Omega \text{ for some }\lambda\in \mathbb C^* \},
\]
where $\Omega$ is the standard torus-invariant volume form and $\Delta$ is the pole divisor of $\Omega$, i.e. the coordinate triangle. A recent work of Loginov and Zhang \cite{loginov2024birationalinvariantsvolumepreserving} examines $\Bir(\PP^n,\Delta)$ in general dimensions and shows that in dimensions $n\geq4$
over $\CC$ and $n\geq 3$ over number fields, $\Bir(\PP^n, \Delta)$ is not generated by pseudo-regularizable elements and therefore not simple.

The group $\mathrm{Bir}(\mathbb{P}^2,\Delta)$ satisfies the Jordan property as it is a subgroup of $\mathrm{Bir}(\mathbb{P}^2)$. We would like to find the Jordan constant for this group. We will show that any finite subgroup $G
\leqslant \Bir(\PP^2, \Delta)$ fits into an exact sequence
\[
1 \to A \to G \to G_D \to 1,
\]
where $G_D$ is roughly the group acting on the components of the pole divisor of the volume form and their intersection strata. More specifically, one may regularize $G$ to be an automorphism subgroup of some rational surface $X$, and $G_D$ is the action on the dual complex of $X$ with the induced volume. One may show that $A$ is abelian. Therefore, to find a Jordan bound $J$, we would like to find a bound on $G_D$, which may be further reduced to examining the relevant Calabi-Yau pairs. 

Altogether, this yields a Jordan constant of $12$ for $\Bir(\Ptwo, \Delta)$, which is smaller than the Jordan constant $7200$ for $\mathrm{Bir}(\mathbb{P}^2)$.

\begin{theorem}\label{thmd}
$J(\mathrm{Bir}(\Ptwo, \Delta))=12$.
\end{theorem}

Based on the Jordan constant $12$ of $\mathrm{Bir}(\mathbb{P}^2, \Delta)$, we establish a Jordan bound for the three-dimensional volume-preserving Cremona group $\mathrm{Bir}(\mathbb{P}^3, \Delta)$.
\begin{theorem} \label{thmthree}
$J(\mathrm{Bir}(\mathbb P^3, \Delta)) \leq 144$.
\end{theorem}

We extend the analysis to the geometric Jordan property, which further requires the abelian normal subgroup to be contained in an algebraic torus. In a recent paper, Moraga discussed the geometric Jordan property for birational automorphism groups in the setting of cluster type varieties and posed the conjecture that the group \(\operatorname{Bir}(\mathbb{P}^n)\) satisfies the weak geometric Jordan property \cite[Conjecture 8.5]{moraga}. In this paper, we show that \(\operatorname{Bir}(\mathbb{P}^2)\) satisfies the geometric Jordan property.

\begin{theorem} \label{thmgeo}
$\mathrm{Bir}(\mathbb P^2)$ satisfies the geometric Jordan property. A weak geometric Jordan bound is $2^{11} \cdot 3^2$. 
\end{theorem}

We fit the analysis of the action on the dual complex in dimension two in our situation to a general setting and we point to more general results in this direction. The notion of regularity of a log-canonical pair $(X,D)$ was introduced by Shokurov \cite{shokurov} and is defined as the dimension of the dual complex $\mathcal{D}(X,D)$ of the boundary divisor $D$, computed on a log resolution. Moraga \cite{mor} defines and studies the dual invariant, the coregularity
\(
\mathrm{coreg}(X)=\dim X-1-\mathrm{reg}(X).
\)
In a recent work, Loginov-Przyjalkowski-Trepalin \cite{przy} extend these notions to varieties endowed with an action of a finite group $G$. By their definition, the $G$-regularity is the maximum of the regularities of all pairs
$(X,D)$, where $D$ is a $G$-invariant log-canonical complement on $X$, i.e., $G\leqslant \Aut(X,D)$ where $\Aut(X,D)$ is the volume-preserving automorphism group. For smooth rational conic bundle surfaces and smooth del Pezzo surfaces, the authors study the dual complex and derive classification results for finite groups according to different values of 
$G$-coregularity. Relating to our situation, saying that a finite automorphism group $G$ on a smooth conic bundle surface or a del Pezzo surface has coregularity zero means that $G\leqslant \Aut(X,D)$ for some Calabi-Yau pair $(X,D)$ of coregularity zero. And since all Calabi-Yau surface pairs of coregularity zero belong to a single crepant birational equivalence class 
\cite[Section 1.3]{ducat2022quarticsurfacesvolumepreserving}, represented by the toric pair $(\mathbb{P}^2,\Delta)$, it follows that $G$ may be viewed as a birational automorphism subgroup of the pair $(\mathbb{P}^2,\Delta)$, and we will see that the converse is also true. In particular, the part of the discussion in Section 3 and 4 about the structure of the group $G$ fits into this general setting in \cite{przy}.

In Section 2, we will review some preliminaries regarding finite subgroups of $\mathrm{Bir}(\Ptwo)$ and volume-preserving maps. In Section 3, we will examine the exact sequence as above, which reduces the task of finding a Jordan bound to examining certain Calabi-Yau pairs on minimal $G$-surfaces. In Section 4, we explicitly find a Jordan bound for $\Bir(\Ptwo, \Delta)$. In Section 5, we show that this Jordan bound is minimal and we present a way to find volume-preserving subgroups regularized on del Pezzo surfaces. In Section 6, we extend the analysis of the Jordan property to the three-dimensional volume-preserving Cremona group $
\Bir(\PP^3, \Delta)$. In Section 7, we find a weak geometric Jordan bound for the plane Cremona group.

\subsection*{Acknowledgments}
I sincerely thank Joaquín Moraga for proposing this project and helpful discussions, the ideas from which essentially form the outline of the paper. I am grateful to Burt Totaro for helpful comments, corrections, and suggestions. I am grateful to Zinovy Reichstein and Federico Scavia for results on torsion index. I sincerely thank Zhijia Zhang for valuable conversation and for the idea of finding volume-preserving subgroups on del Pezzo surfaces. I am grateful to Konstantin Loginov, Victor Przyjalkowski, and Andrey Trepalin for helpful comments, suggestions, and simplification of the proofs of groups acting on dual complexes. I am grateful to Ko Honda, Raphaël Rouquier, Jas Singh and José Yáñez for helpful comments.
\newpage
\section{Preliminaries}
We work over the field of complex numbers.
\subsection{Finite subgroups of $\boldsymbol{\mathrm{Bir}(\mathbb{P}^2)}$}

In \cite{dolgachev2009finitesubgroupsplanecremona}, Dolgachev and Iskovskikh describe finite subgroups of the plane Cremona group $\mathrm{Bir}(\mathbb{P}^2)$. The procedure goes as follows:
\begin{itemize}
    \item Any finite subgroup $G \leqslant \mathrm{Bir}(\mathbb{P}^2)$ can be regularized on a rational surface $X$ such that $G$ may be viewed as a finite subgroup of $\mathrm{Aut}(X)$. That is, there exists a birational map $f: X \dashrightarrow \mathbb{P}^2$ such that $f^{-1}Gf \leqslant \mathrm{Aut}(X)$.
    \item Performing a $G$-equivariant minimal model program to $X$, we get a minimal $G$-surface $X'$. One can show that such a minimal surface is either a conic bundle or a del Pezzo surface.
    \item For each minimal $G$-surface, the authors classify the automorphism group of the surface and its finite subgroups, and therefore one gets a list of finite subgroups of the plane Cremona group.
\end{itemize}
\begin{definition}[{\cite[Section 3.2]{dolgachev2009finitesubgroupsplanecremona}}]
Let $G$ be a finite group. A \emph{$G$-surface} is a pair $(S,\rho)$, where $S$ is a
nonsingular projective surface and $\rho$ is an isomorphism from $G$ to a group of
automorphisms of $S$.

A morphism of $G$-surfaces
\(
(S,\rho) \rightarrow (S',\rho')
\)
is a morphism of surfaces $f \colon S \to S'$ such that
\(
\rho'(G) = f \circ \rho(G) \circ f^{-1}.
\)
\end{definition}
\begin{definition}[{\cite[Definition~3.7]{dolgachev2009finitesubgroupsplanecremona}}]
A \emph{minimal $G$-surface} is a $G$-surface $(S,\rho)$ such that any birational
morphism of $G$-surfaces
\(
(S,\rho)\rightarrow (S',\rho')
\)
is an isomorphism. 
\end{definition}

\begin{theorem}[{\cite[Theorem~3.8]{dolgachev2009finitesubgroupsplanecremona}}]
Let $S$ be a minimal rational $G$-surface. Then either
\begin{enumerate}
  \item $S$ admits a structure of a conic bundle with $\Pic(S)^G \cong \mathbb Z^2$, or
  \item $S$ is isomorphic to a del Pezzo surface with $\Pic(S)^G \cong \mathbb Z$.
\end{enumerate}
\end{theorem}

\begin{remark}

In case~(1), the action of $G$ on $S$ preserves the fibration $S \to \mathbb P^1$,
so that $G \leqslant \Aut(S,\pi)$. A conic bundle minimal $G$-surface is either isomorphic to the Hirzebruch surface $F_n$ or to a surface obtained
from $F_n$ by blowing up a finite set of points, no two lying in a fiber of the
ruling; in particular, each singular fiber has exactly two irreducible components. Del Pezzo surfaces of degree $7$ are not $G$-minimal, since they admit a $G$-equivariant contraction of a $(-1)$-curve. In degree $8$, there are two del Pezzo surfaces, namely $\mathbb P^1 \times \mathbb P^1$
and the Hirzebruch surface $F_1$.
The surface $F_1$ contains a unique $(-1)$-curve, which is invariant under
all automorphisms; hence $F_1$ is not $G$-minimal. As a consequence, a minimal del Pezzo $G$-surface is either isomorphic to
$\mathbb P^2$, $\mathbb P^1 \times \mathbb P^1$, or of degree $d \le 6$.
\end{remark}

\subsection{Volume-preserving map}
\begin{definition}
A \emph{volume form} $\omega_X$ on a normal projective variety $X$ is a nonzero top degree rational differential form.  
\end{definition}

\begin{definition}
A birational morphism $f: (X,\omega_X) \dashrightarrow (Y,\omega_Y)$ is \emph{volume-preserving} if $f^*\omega_Y = \lambda \omega_X$ for some $\lambda \in \mathbb{C}^*$.

\end{definition}

\begin{definition}
A \emph{Calabi-Yau pair} $(X,D_X)$ consists of a \emph{normal projective} variety $X$
and an effective integral Weil divisor $D_X$ such that
\(
K_X + D_X \sim 0,
\)
and the pair $(X,D_X)$ is log canonical.
There exists a rational volume form $\omega_X$ on $X$, unique up to
multiplication by a nonzero constant, such that
\(
D_X + \mathrm{div}(\omega_X)=0.
\)
\end{definition}

\begin{remark}
More generally, one may define a Calabi-Yau pair by allowing higher index, i.e.\
requiring only that $K_X + D \sim_{\mathbb Q} 0$; in this paper, however, we work
exclusively with the index~$1$ case and simply refer to such pairs as
Calabi-Yau pairs.
\end{remark}

\begin{definition}[{\cite[Definition 1.6, Remark 1.7]{corti}}]\label{crep}
A birational map of Calabi-Yau pairs $f: (X, D_X) \dashrightarrow (Y, D_Y)$ is \emph{volume-preserving} if the corresponding $f: (X, \omega_X) \dashrightarrow (Y, \omega_Y)$ is volume-preserving. In particular, if $f\colon (X,D_X)\to (Y,D_Y)$ is a birational morphism, then $f$ is volume-preserving if and only if $f_*(D_X) = D_Y$. We say two pairs $(X, D_X)$ and $(Y, D_Y)$ are \emph{crepant birational} if there is a volume-preserving map $f: (X, D_X) \dashrightarrow (Y, D_Y)$.
\end{definition}

\begin{definition}
The \emph{standard torus-invariant volume form} on $\mathbb{P}^n$ with coordinates $x_0,x_1,...,x_n$ is given by the formula
\[
\Omega = \frac{dx_1}{x_1} \wedge \frac{dx_2}{x_2}\wedge ...\wedge \frac{dx_n}{x_n}
\]
in the affine chart $\{{x_0= 1}\}$. We let $\Delta$ denote $-\mathrm{div}(\Omega)$. 
Let $\mathrm{Bir}(\mathbb{P}^n, \Delta)$ denote the volume-preserving birational automorphisms of $\mathbb{P}^n$ with respect to $\Delta$. In dimension 2, the pole divisor $\Delta$ is the coordinate triangle in $\mathbb{P}^2$. An example of a volume-preserving birational map of $\mathbb{P}^2$ is the standard quadratic transformation $[x_0,x_1,x_2] \mapsto [x_1x_2,x_0x_2,x_0x_1]=[1/x_0,1/x_1,1/x_2]$.
\end{definition}

\begin{notation} We use $\Omega$ for the standard torus-invariant volume form on $\mathbb{P}^n$ and let $\Delta = -\mathrm{div}(\Omega)$. For a general variety $X$, we use $\omega_X$ to denote a volume form, and we let $D_X = -\mathrm{div}(\omega_X)$. Also, we refer to $D_X = -\mathrm{div}(\omega_X)$ as the pole divisor, although it may differ from the usual pole divisor. We use $\Aut(X,\omega_X)$ and $\Bir(X,\omega_X)$ to denote the group of volume-preserving automorphisms and birational automorphisms of $X$ with volume form $\omega_X$. In the case of Calabi-Yau pairs, we also use the notation $\Aut(X,D_X)$ and $\Bir(X,D_X)$.
\end{notation}

\begin{remark}
For a birational map $f: X \dashrightarrow (Y, \omega_Y)$, we can endow $X$ with a volume form $\omega_X = f^*\omega_{Y}$ such that the map is volume-preserving. The form $\omega_X$ is called the \emph{pullback volume form}. Note that the pullback of a form with effective pole divisor, i.e., a form coming from a Calabi-Yau pair, may not have effective pole divisor.
\end{remark}

\begin{example}
    
We give the pullback volumes of $\Delta$ under some blow-ups of $(\PP^2, \Delta)$. The right-hand side of the picture shows $\Delta$ and the left-hand side is the pole divisor $D_{F_1}$ of the pullback volume form, where $F_1 = \mathrm{Bl}_p \PP^2$. Note that $D_{F_1}$ may not be effective.

Blowing up at a point of multiplicity 2 introduces poles along the exceptional divisor $E$ for the pullback volume form, i.e., the pole divisor of the pullback has an additional component which is the exceptional divisor:
\begin{figure}[H]
\centering
\includegraphics[width=7.5cm]{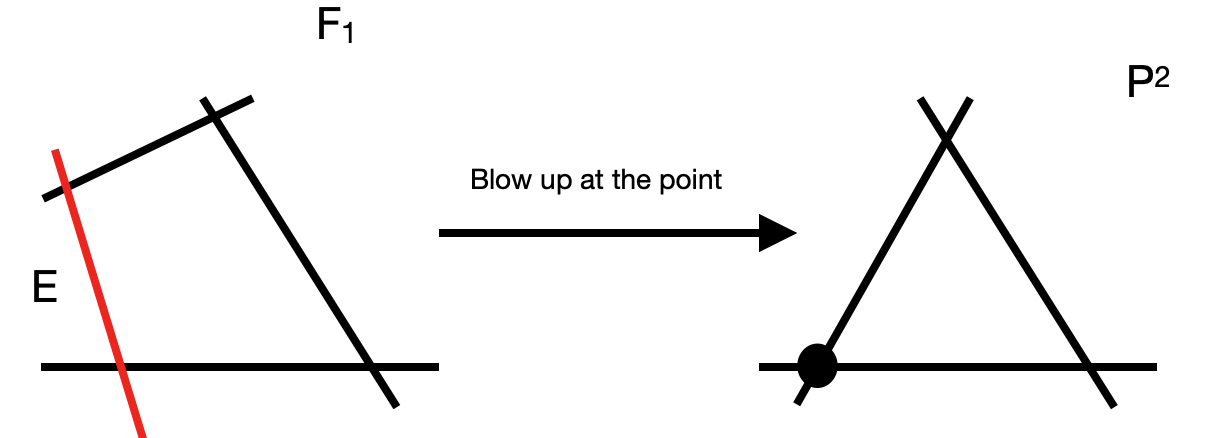}
\end{figure}

Blowing up at a point of multiplicity 1 does not introduce any new component to the pole divisor of the pullback volume form:
\begin{figure}[H]
\centering
\includegraphics[width=7cm]{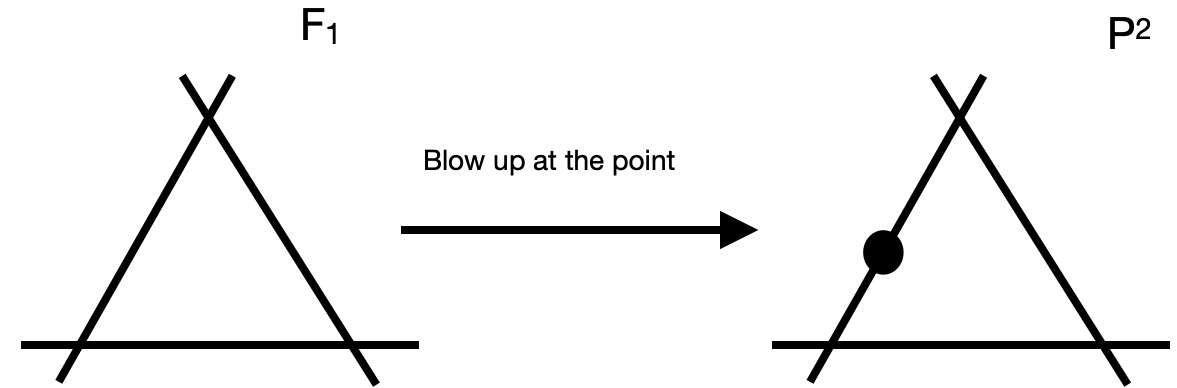}
\end{figure}
Blowing up at a point outside the pole divisor introduces zeros along the exceptional divisor $E$ for the pullback volume form, i.e., the pole divisor of the pullback has a negative part on the exceptional divisor:

\begin{figure}[H]
\centering
\includegraphics[width=7cm]{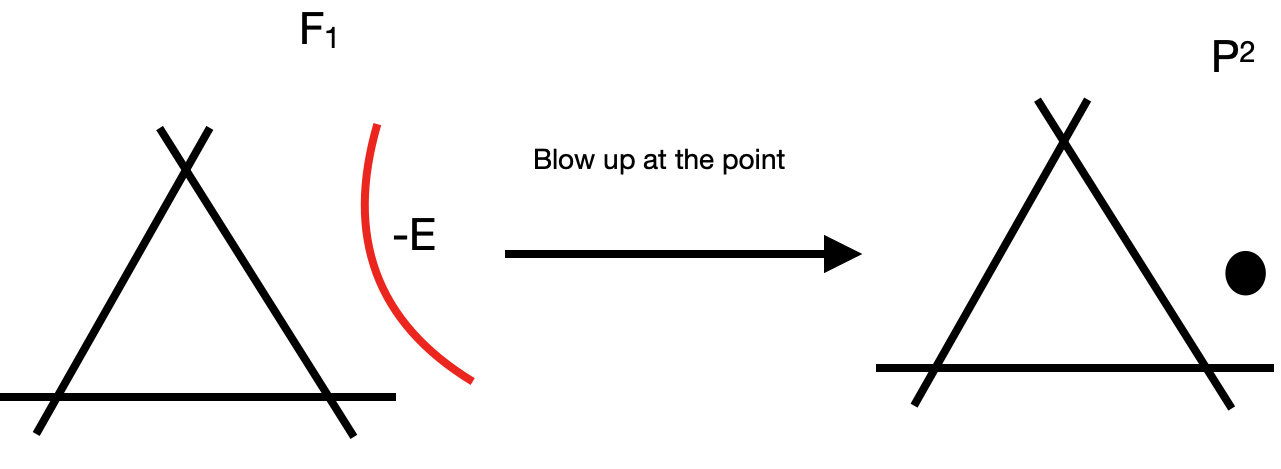}
\end{figure}
\end{example}

\section{Action on the dual complex}

Given a finite volume-preserving subgroup $G \leqslant \mathrm{Bir}(\mathbb{P}^2, \Delta)$, we would like to derive an exact sequence $1 \to A \to G \to G_D \to 1$, where $G_D$ is roughly the group acting on the components of the pole divisor of the volume form and their intersection strata, and $A$ is the subgroup that fixes all components and their intersection strata. To do that:
\begin{itemize}
    \item We need to regularize $G$ on a rational surface $X$ such that $G \leqslant \mathrm{Aut}(X, \omega_X)$ for some volume form $\omega_X$ on $X$. 
    \item We would like $D_X = -\mathrm{div}(\omega_X)$ to be effective for the following reason: If this is true, then $(X,D_X)$ is a Calabi-Yau pair. Calabi-Yau pairs are better to classify than non-effective pairs.
    \item We would like to show that $A$ is abelian, so that finding a Jordan bound $J$ will be reduced to finding a bound for the group $G_D$. After classifying all Calabi-Yau pairs, we examine the possible groups acting on the components of the pole divisors and their intersection strata, which give a bound for $G_D$. 
\end{itemize}

We address the above by the following lemmas.
\begin{lemma} If $G$ is a finite subgroup of $ \mathrm{Bir}(\mathbb{P}^2,\Delta)$, then $G\leqslant \mathrm{Aut}(X,\omega_X)$ for some minimal $G$-surface $X$ and some volume form $\omega_X$ on $X$.
\end{lemma}
\begin{proof}
Any finite subgroup $G\leqslant \Bir(\PP^2,\Delta)$ can be regularized on a minimal rational $G$-surface $X$ \cite[Lemma 3.5]{dolgachev2009finitesubgroupsplanecremona}. Given the regularization map $f: X\dashrightarrow \Ptwo$, we can define $\omega_X = f^*\Omega$, then $G\leqslant \mathrm{Aut}(X,\omega_X)$.
\end{proof}
Note that this $\omega_X$ is not guaranteed to have no zeros, therefore $D_X$ may not be effective. We show the following.
\begin{lemma}\label{lemm}
If we have a finite subgroup $G\leqslant \mathrm{Bir}(\mathbb{P}^2, \Delta)$, then it can be regularized on a minimal $G$-surface $X$ under some birational map $f: X \dashrightarrow \mathbb{P}^2$ such that $\omega_X = f^*\Omega$ has no zeros.
\end{lemma}

\begin{proof}
Let $\varphi: \mathbb{P}^2 \dashrightarrow \mathbb{P}^2$ be a volume-preserving map, and consider the minimal resolution of indeterminacy
\[
\begin{tikzcd}
& X \arrow[ld, "p"'] \arrow[rd, "q"] & \\
\mathbb{P}^2 \arrow[rr, dashed, "\varphi"] & & \mathbb{P}^2 .
\end{tikzcd}
\]
We claim that the pullback volume form on $X$ along $p$ has no zeros along the exceptional divisors, and hence no zeros at all. Indeed, suppose that $E$ is an exceptional divisor of $p$ along which the pullback form vanishes, then every exceptional divisor lying above $E$ also carries zeros. This follows from that at each step the pole divisor remains simple normal crossing with all coefficients $\leq 1$: by induction if we blow up the intersection of two components with coefficients $a$ and $b$, then the new exceptional divisor has coefficient $a+b-1 \leq 1$; if we blow up a smooth point on a component with coefficient $a$, then the new exceptional divisor has coefficient $a-1 \leq 1$; and if we blow up a point outside the boundary, then the new exceptional divisor has coefficient $-1 \leq 1$. So if an exceptional divisor has negative coefficient $a$, then blowing up on a point on the exceptional divisor will lead to coefficient $\leq a+1-1 < 0$. Hence every exceptional divisor lying above $E$ also carries zeros. Let $E_{\max}$ be a maximal exceptional divisor lying above $E$, i.e. no exceptional divisor lies above \(E_{\max}\). By \cite[Lemma 2.11(2)]{deFernexEin2002}, $E_{\max}$ is not contracted by $q$, so $q_* p^* \Omega$ would vanish along $q_* E_{\max}$. This is impossible, since
\(
q_* p^* \Omega \;=\; \varphi^{-1*}\Omega \;=\; \Omega,
\)
as $\varphi^{-1}$ is volume-preserving and $\Omega$ has no zeros.

Now let $G \leqslant \mathrm{Bir}(\PP^2, \Delta)$ be a finite subgroup. To find a regularization of $G$, by \cite{deFernexEin2002} we resolve the union of the minimal indeterminacy loci of all maps $\varphi$ in $G$. Each such resolution does not introduce any zeros to the pullback volume form, so the resulting model still carries a nowhere vanishing volume form.
To obtain a minimal $G$-surface, one runs a $G$-equivariant minimal model program by blowing down $(-1)$-curves. This process does not introduce any zeros to the volume form.
\end{proof}

We show a few more lemmas about the dual complex of such pairs $(X, D_X)$ from the above, i.e., Calabi-Yau pairs obtained from $(\Ptwo,\Delta)$ through a sequence of blow-ups at points on $\Delta$ and blow-downs along $(-1)$-curves.

\begin{lemma}\label{cycle}
$D_X$ is either a cycle of $\mathbb{P}^1$'s meeting transversely, or an irreducible rational curve with one nodal point.
\end{lemma}

\begin{proof}
Since $\Delta$ is a cycle of $\PP^1$'s, blowing up points on $\Delta$ preserves the cycle structure. Let $C$ be a $(-1)$-curve on $X$. Then $C \cdot (-K_X) = C^2 + 2 = 1$. If $C$ is not a component of $D_X$, then $C$ intersects with $D_X$ at one point, so contracting $C$ does not change the configuration of $D_X$. If $C$ is a component of the divisor $D_X$, then contracting $C$ will either make $D_X$ remain a cycle of smooth $\PP^1$'s or, in the case that there are only two components, then contracting one component will result in a rational curve with one node.
\end{proof}

\begin{definition}[{\cite[Section 2.2]{loginov2024birationalinvariantsvolumepreserving}}]
Let \( D_X = \sum D_i \) be a Cartier divisor on a smooth variety \( X \).
The \emph{dual complex}, denoted by \(\mathcal{D}(D_X)\), of a simple normal crossing divisor
\( D_X = \sum_{i=1}^r D_i \) on a smooth variety \( X \) is a CW-complex constructed as follows.
The simplices \( v_Z \) of \(\mathcal{D}(D_X)\) are in bijection with irreducible components
\( Z \) of the intersection \(\bigcap_{i \in I} D_i\) for any non-empty subset
\( I \subset \{1, \ldots, r\} \), and the vertices of \( v_Z \) correspond to the components
\( D_i \) with \( i \in I \). In particular, the dimension of \( v_Z \) is equal to
\(\# I - 1\). We call \( Z \) a \emph{stratum} of \( D_X \).

The gluing maps are constructed as follows. For any non-empty subset
\( I \subset \{1, \ldots, r\} \), let \( Z \subset \bigcap_{i \in I} D_i \) be a stratum,
and for any \( j \in I \), let \( W \) be the unique component of
\(\bigcap_{i \in I \setminus \{j\}} D_i\) containing \( Z \).
Then the gluing map is the inclusion of \( v_W \) into \( v_Z \) as a face of \( v_Z \)
that does not contain the vertex \( v_i \) corresponding to \( D_i \).
Note that the dimension of \(\mathcal{D}(D_X)\) does not exceed \(\dim X - 1\).
If \(\mathcal{D}(D_X)\) is empty, i.e.,\ \( D_X = 0 \), we say that
\(\dim \mathcal{D}(D_X) = -1\).

We denote by \( D_X^{=1} \) the sum of the components of \( D_X \) with coefficient \( 1 \).
For an lc pair \((X, D_X)\), we define \(\mathcal{D}(X, D_X)\) as
\(\mathcal{D}(D_Y^{=1})\), where \( f \colon (Y, D_Y) \to (X, D_X) \) is a log resolution
of \((X, D_X)\), so that we have
$K_Y + D_Y = f^*(K_X + D_X).$ It is known that the PL–homeomorphism type of 
\(\mathcal{D}(D_Y^{=1})\) is independent of the choice of log resolution; 
see \cite[Proposition~11]{tommaso}.
\end{definition}

\begin{lemma}
The dual complex of $(X, D_X)$ is homeomorphic to $S^1$.
\end{lemma}
\begin{proof}
Follows from Lemma \ref{cycle}. See \cite[Lemma 1.6.3]{przy} for a more general statement.
\end{proof}

Let $X$ be a del Pezzo surface or a conic bundle and $(X,D_X)$ a Calabi-Yau pair crepant birational to $(\Ptwo, \Delta)$. Given a finite $G \leqslant \mathrm{Aut}(X, D_X)$ minimal on $X$, we see that $G$ acts on $\mathcal D(X,D_X)$, so we have an exact sequence
\[
1 \to A \to G \to G_D \to 1
\]
where $G_D$ is the group acting on the dual complex $\mathcal D(X,D_X)$.

\begin{lemma}[{\cite[Proposition 2.5.2]{przy}}]\label{abelian}
$A$ is abelian of rank at most 2.
\end{lemma}

\section{Jordan constant for $\mathrm{Bir}(\PP^2, \Delta)$}
For any finite $G\leqslant \Bir(\Ptwo,\Delta)$, we regularize $G\leqslant \Aut(X,D_X)$, and we have established the exact sequence $1 \to A \to G \to G_D \to 1$.
Since $A$ is abelian normal, to find a Jordan bound $J$ for $\mathrm{Bir}(\mathbb{P}^2, \Delta)$, it suffices to find a bound for the size of $G_D$.

\begin{lemma}
For a minimal del Pezzo $G$-surface $X$ of degree $d\leq 6$, if $G\leqslant \Aut(X,D_X)$ for some $(X,D_X)$ crepant birational to $(\Ptwo, \Delta)$, then we have that $|G_D| \leq |D_d|= 2d$.\\
\end{lemma}
\begin{proof}
Write
\(
D_X=\sum_{i=1}^k D_i
\)
as a cycle of $\mathbb P^1$'s, we have
\(
d = (-K_X)\cdot D_X = \sum_{i=1}^k (-K_X)\cdot D_i.
\)
Since $-K_X$ is ample, the number of irreducible components of $D_X$ is bounded by $d$. By Lemma \ref{cycle} the boundary divisor $D_X$ is either a cycle of smooth rational curves or an irreducible nodal rational curve. In the case of $ d\geq 2$, the order of the group on the dual complex is bounded by $|D_d|$. For \(d=1\), after blowing up the node of \(D_X\), the induced \(G\)-action preserves both the exceptional divisor and the strict transform of \(D_X\), so the action on the dual complex may only interchange the two edges, and hence \(|G_D| \le |D_1| = 2\).
\end{proof}

\begin{lemma}\label{conic}
For a minimal conic bundle $G$-surface $X$, if $G\leqslant \Aut(X,D_X)$ for some $(X,D_X)$ crepant birational to $(\Ptwo, \Delta)$, then we have that $|G_D| \leq |D_2|=4$.
\end{lemma}

\begin{proof}
Since a general fiber $F\simeq \PP^1$ satisfies $(-K_X)\cdot F=2$, we have \(D_X\cdot F=2\), hence the horizontal part of $D_X$ consists of at most two components. Since $G$ preserves the set of horizontal components, we have that the order of the group on the dual complex is bounded by $|D_2|$.
\end{proof}

\begin{remark}
See \cite{przy} for a general discussion on the structure of a finite group $G$ with arbitrary coregularity on a smooth rational conic bundle surface or a del Pezzo surface. In particular, for any coregularity zero $G$ acting on Calabi-Yau pairs with index not necessarily one, the action on the dual complex is either trivial, cyclic, or a dihedral group \cite[Lemma 2.5.3]{przy}.
\end{remark}

\begin{theorem}
$J(\mathrm{Bir}(\Ptwo, \Delta))=12$.
\end{theorem}

\begin{proof}
The previous two lemmas establish the bound for del Pezzo surfaces of degree $d\leq 6$ and conic bundles. For the two remaining cases of minimal $G$-surfaces $\Ptwo$ and $\PP^1\times \PP^1$ with invariant Picard rank $\rho^G = 1$: If \(X\simeq \Ptwo\), then $D_X$ is either a triangle of lines, an irreducible nodal cubic, or a conic and a line. So the order of the group on the dual complex is bounded by \(|D_3|\). If \(X\simeq \PP^1\times\PP^1\), then any \(D_X\) has at most four irreducible components, hence the group is bounded by \(|D_4|\). 

Moreover, the bound $12 = |D_6|$ is minimal because for the del Pezzo surface of degree $6$, we exhibit a finite volume-preserving automorphism subgroup $G$ whose induced action on the dual complex is $D_6$, and moreover, the group does not admit any abelian normal subgroup of index smaller than $12$, as demonstrated in Lemma \ref{grp} in the next section.
\end{proof}

\section{Volume-preserving subgroups of del Pezzo surfaces}

We show that the Jordan bound $J = 12$ for $\Bir(\Ptwo, \Delta)$ is indeed minimal by exhibiting a volume-preserving subgroup regularized on the del Pezzo surface of degree 6 that contains only abelian normal subgroups of index at least $12$. We also mention a way to potentially find all finite volume-preserving subgroups regularized on del Pezzo surfaces with degree $d>2$, and we thank Zhijia Zhang for this idea.
\subsection{Jordan constant of $\Bir(\Ptwo, \Delta)$}
We will see that the Jordan bound $12$ is obtained by a volume-preserving subgroup regularized on the del Pezzo surface of degree 6. See full list of finite subgroups of $\Aut(X)$ in \cite[Theorem 4.7]{dolgachev2009finitesubgroupsplanecremona}.

\begin{example}
Let $X$ be the del Pezzo surface of degree $6$, obtained by blowing up the three
coordinate points of $\mathbb P^2$, and let $D_X$ be the induced boundary divisor
given by the strict transforms of the coordinate lines together with the exceptional
divisors. Then $D_X$ is a cycle of six $\mathbb P^1$'s, and $(X,D_X)$ is a Calabi-Yau
surface pair.
Let $n\ge 2$, and let $\epsilon_n$ denote a primitive $n$th root of unity.
Consider the finite subgroup $G_0\leqslant \Aut(X)$ generated by the following
automorphisms of $\mathbb P^2$, lifted to $X$:
\[
\begin{aligned}
\alpha_1 &\colon [x_0,x_1,x_2] \mapsto [\epsilon_n x_0,\, x_1,\, x_2],\\
\alpha_2 &\colon [x_0,x_1,x_2] \mapsto [x_0,\, \epsilon_n x_1,\, x_2],\\
s_2 &\colon [x_0,x_1,x_2] \mapsto [x_0,\, x_2,\, x_1],\\
s_3 &\colon [x_0,x_1,x_2] \mapsto [x_2,\, x_0,\, x_1].
\end{aligned}
\]
Then $G_0 \cong (\mathbb Z/n\mathbb Z)^2 \rtimes S_3$, where $S_3$ acts by permuting
the coordinates.
Let $s_1\in \Aut(X)$ be the lift of the standard quadratic transformation
\[
[x_0,x_1,x_2] \mathrel{\mapstochar\dashrightarrow} [x_1x_2,\, x_0x_2,\, x_0x_1],
\]
which is a biregular involution of $X$.
We consider the finite subgroup $G = \langle G_0, s_1\rangle \leqslant \Aut(X)$.
This group acts by volume-preserving
automorphisms of $(X,D_X)$.
The induced action on the dual complex $\mathcal D(X,D_X)$ factors through a
dihedral group $G_D \cong D_6$ of order $12$, generated by the permutations
$s_2,s_3$ together with the involution induced by $s_1$, which acts as a rotation
by $180^\circ$ on the cycle $D_X$.
We will show that for some $n$, $G$ does not admit a normal abelian subgroup
of index strictly less than $12$, which implies that the Jordan bound $12$ in this
case is minimal.
\end{example}

\begin{lemma} \label{grp}
For $\mathrm{gcd}(n,6)=1,n\geq 2$, we have that $G$ does not admit a normal abelian subgroup of index smaller than $12$.
\end{lemma}
\begin{proof}
Let 
\[
1 \rightarrow A \rightarrow G \xrightarrow{\phi} D_6 \rightarrow 1
\]
be the decomposition, where $A:= (\mathbb{Z}/n\mathbb{Z})^2$. We have a linear representation 
\[
\rho: D_6 \rightarrow \mathrm{Aut}(A) \simeq \mathrm{GL}_2(\mathbb{Z}/n\mathbb{Z}).
\]

The $3$–cycle $(1\,2\,3)\in S_3 \leqslant D_6$ (the $120^\circ$ rotation) acts on $A$ by the matrix $U$
\[
U = \begin{pmatrix} -1 & 1 \\ -1 & 0 \end{pmatrix}, \qquad \det(U - I) = 3.
\]
The element $s_1$ (the $180^\circ$ rotation) acts by the matrix $Z$
\[
Z = -I, \qquad \det(Z - I) = 4.
\]
Over the ring $\mathbb{Z}/n\mathbb{Z}$, a matrix is invertible if and only if its determinant is a unit, which means that the determinant and $n$ are relatively prime. Since $\gcd(n,6)=1$, both $3$ and $4$ are invertible in $\mathbb{Z}/n\mathbb{Z}$, so the endomorphisms $U - I$ and $Z - I$ are invertible on $A$.

If $x \in G$ maps to $g \in D_6$ under $\phi$, then $x$ acts on $A$ by $g$, and for any $a \in A$ we have
\[
[x,a] = x a x^{-1} a^{-1} = (g \cdot a) - a = (g - I)a.
\]
Therefore $[x,A] = (g - I)A$, and if $g - I$ is invertible then $[x,A] = A$.

Now let $N \lhd G$ be an abelian normal subgroup. If $\phi(N)\neq 1$, then $\phi(N)$ is a nontrivial normal abelian subgroup of $D_6$. Thus $\phi(N)$ contains either the $120^\circ$ rotation $r^2$ or the $180^\circ$ rotation $r^3$, where $r$ denotes the $60^\circ$ rotation. Choose $x \in N$ mapping to such an element $g=r^2$ or $r^3$. Since $g - I$ is invertible, we obtain $[x,A] = A \leqslant N$. Abelianness then forces $x$ to centralize $A$, i.e.,\ $(g - I) = 0$ on $A$, which contradicts the invertibility of $g - I$. Hence $\phi(N)=1$ and $N \leqslant A$.

Consequently,
\[
[G:N] \ge [G:A] = |D_6| = 12.
\]
Therefore, $G$ admits no abelian normal subgroup of index smaller than $12$, and equality occurs only when $N = A$.
\end{proof}

\subsection{Volume-preserving automorphism subgroups of del Pezzo surfaces}
Let $X$ be a del Pezzo surface of $d>2$, given a finite subgroup $G\leqslant \mathrm{Aut}(X)$, we determine if $G$ is a subgroup of $\mathrm{Aut}(X,D_X)$, where $(X,D_X)$ is some Calabi-Yau pair crepant birational to $(\Ptwo, \Delta)$. 

For a del Pezzo surface of degree $d>2$, we know that the anti-canonical divisor is very ample \cite{Dolgachev2012}. For those surfaces, we consider the anti-canonical embedding given by \(|-K_X|\) into \(\mathbb{P}^d\). We will see that the $G$-action on $X$ induces an action on $\mathbb P(H^0(X,-K_X)^*) \simeq \mathbb P^d$ and a $G$-invariant effective divisor in  \(|-K_X|\) corresponds to a hyperplane section of a $G$-invariant hyperplane of $\PP^d$.

\begin{lemma}
Let \(X\) be a del Pezzo surface with \(d > 2\), $G\leqslant \Aut(X)$, and \(f: X \rightarrow \mathbb{P}^d = \mathbb{P}(H^0(X, -K_X)^*)\) be the embedding. Then,  \(X\) has a $G$-invariant \(D_X \in |-K_X|\) iff the action on \(\mathbb{P}(H^0(X, -K_X)^*)\) has a $G$-invariant hyperplane.
\end{lemma}

\begin{proof}
$G$ acts on sections of \(-K_X\) by \(g \cdot s(x) = s(g^{-1} x)\), so $G$ acts on \(\mathbb{P}(H^0(X, -K_X)^*)\). The image of \(D_X\in |-K_X|\) in \(\mathbb{P}(H^0(X, -K_X)^*)\) is the intersection of the image of \(X\) with a hyperplane. We have $G$-equivariant restriction map \(H^0(\mathbb{P}(H^0(X, -K_X)^*), O(1)) \rightarrow H^0(X, -K_X)\). Note that \(H^0(\mathbb{P}(H^0(X, -K_X)^*), O(1))\cong H^0(X, -K_X)^{**} \cong H^0(X, -K_X)\) so the restriction map is an isomorphism. So if $D_X$ is invariant, since there exists a unique hyperplane $H$ which restricts to $D_X$, we have that $H$ is invariant.
\end{proof}
A Calabi-Yau pair $(X,D_X)$ with $X$ smooth and rational is crepant birational to $(\Ptwo, \Delta)$ if and only if $D_X$ is singular with nodal singularities \cite[Section 1.3]{ducat2022quarticsurfacesvolumepreserving}. Therefore, a finite subgroup $G\leqslant \Bir(\Ptwo)$ on a del Pezzo surface $X$ with $d>2$ is volume-preserving iff the $G$-representation on $\mathbb C^{d+1}$ is reducible having an invariant subspace of dimension \(d\), and the intersection of the hyperplane with $X$ has nodal singularities. Both conditions can be checked using Magma \cite{magma}.
\vspace{0.3cm}

We illustrate this method for the degree 5 del Pezzo surface. In this case, \(X\) is isomorphic to \(\mathbb{P}^2\) blown up at \(p_1 = [1, 0, 0]\), \(p_2 = [0, 1, 0]\), \(p_3 = [0, 0, 1]\), and \(p_4 = [1, 1, 1]\). 

We follow \cite{dolgachev2009finitesubgroupsplanecremona}. Consider the lift of the standard quadratic transformation \(\tau_1\) and \(\phi\) a projective transformation that takes \(p_1\) to \(p_1\), \(p_2\) to \(p_2\), and \(p_4\) to \(p_3\). Then \(\tau_2 :=\phi^{-1} \tau_1 \phi\) is not defined at \(p_1\), \(p_2\), and \(p_4\) but fixes \(p_3\). Similarly, we can construct \(\tau_3\) and \(\tau_4\) that fix \(p_2\) and \(p_1\). The automorphism group of $X$ is given by 
$
\mathrm{Aut}(X)\cong S_5 \cong \langle \tau_1, \tau_2, \tau_3, \tau_4 \rangle
$ by sending $\tau_i$ to $(1,i+1)$.
\begin{theorem}[{\cite[Theorem 6.4]{dolgachev2009finitesubgroupsplanecremona}}]
Let $G$ be a finite automorphism subgroup of a $G$-minimal del Pezzo surface of degree $d = 5$. Then $G = S_5, A_5, 5:4, 5:2, \text{ or }C_5$.
\end{theorem}

\begin{proposition}
$S_5$, $A_5$, $5:4$ cannot be realized as a volume-preserving subgroup. $5:2$, $C_5$ can be realized as a volume-preserving subgroup.
\end{proposition}
\begin{proof}
We can embed $X$ into $\mathbb{P}^5$. The action of $\mathrm{Aut}(X)\cong S_5$ is described in \cite[Theorem 0.1]{bauer2020mathfraks5equivariantsyzygiesdel}. We let $s_{ij}$ with $ i\neq j, i,j=1,2,3$ be a basis of $\mathbb C^6$. The action of $\tau_1=(1,2)$ is given by $s_{ij} \mapsto s_{\tau_1(i)\tau_1(j)}$ and the action of $(1,2,3,4,5)$ is given by \[
\begin{aligned}
s_{12} &\mapsto s_{31}, \quad s_{13} \mapsto s_{13} - s_{31} - s_{23}, \quad s_{21} \mapsto s_{21}, \\
s_{23} &\mapsto  s_{12} - s_{21} - s_{32}, \quad s_{31} \mapsto  s_{13} - s_{31} + s_{21}, \quad s_{32} \mapsto s_{12} - s_{21} + s_{31}.
\end{aligned}
\] Inputting the representation of each group in Magma, we get that only $5:2$ and $5$ have  irreducible one-dimensional sub-representations. For $5:2$ and $5$, one can realize the group as the following. Consider $5 = \langle(13452)\rangle$ and $5:2 = \langle (1\,3\,4\,5\,2),\; (3\,2)(4\,5) \rangle$. Both groups are volume-preserving.
\end{proof}

\section{Jordan bound for $\mathrm{Bir}(\mathbb P^3, \Delta)$}

We will use the Jordan constant $12$ for $\mathrm{Bir}(\mathbb P^2, \Delta)$ to derive a Jordan bound for $\mathrm{Bir}(\mathbb P^3, \Delta)$.

Before showing the arguments, we have to show that any finite volume-preserving subgroup $G\leqslant \mathrm{Bir}(\mathbb P^3, \Delta)$ admits a regularization on a (possibly singular) rational three-dimensional Calabi-Yau pair. In Lemma \ref{lemm}, we show the statement for dimension $2$, the arguments that we use are specific to dimension $2$. We provide a proof for general dimensions. I thank Joaquín Moraga for the proofs of the following two statements. 

\begin{lemma}
Let $X$ be a normal $\mathbb{Q}$-factorial projective variety. Let $E$ be an effective divisor on $X$. Let 
$X \dashrightarrow Y$ be a birational contraction to a projective normal variety. Assume that for every prime component 
$P$ of $E$, the center of $P$ on $Y$ is not a divisor. Then, we have that $\mathrm{supp} E \subset \mathrm{Bs}_{-}(E)$.
\end{lemma}

\begin{proof} Let $p : Z \to X$ and $q : Z \to Y$ be a common resolution of $X \dashrightarrow Y$.
Let $H$ be an ample divisor on $X$ and let $F$ be an effective divisor on $Z$ which is
antiample over $Y$ and $q$-exceptional. Let $0 < \delta < 1$. Set
\[
D := p^*H - \delta F - p^*E.
\]
Note that
$
q_*D = q_*(p^*H - \delta F - p^*E) = q_*p^*H
$
is an effective divisor on $Y$ and
$
-D = -p^*H + \delta F + p^*E
$
is ample over $Y$. By the
negativity lemma we conclude that $D$ is effective. Pushing forward by $p$, we obtain
$H - E = p_*D \ge 0.$ Since $H$ was an arbitrary ample divisor on $X$, this implies that every prime component of $E$ belongs to the diminished base locus $\mathrm{Bs}_{-}(E)$; that is,
$\operatorname{supp} E \subset \mathrm{Bs}_{-}(E)$. 
\end{proof}

\begin{theorem}
Let $(X,D_X)$ be a log Calabi-Yau pair. Let $G \leqslant \mathrm{Bir}(X,D_X)$ be a finite subgroup. Then, there is a log Calabi-Yau pair $(X',D_X')$ which is crepant birational equivalent to $(X,D_X)$ for which $G \leqslant \mathrm{Aut}(X',D_X')$.
\end{theorem}

\begin{proof}
Let $X''$ be a birational model of $X$ on which $G$ acts by automorphisms. By passing to a higher $G$-equivariant birational model of $X''$, we may assume that there is a birational contraction $p: X'' \to X$. We may assume that $X''$ is $\mathbb{Q}$-factorial. Let $p^*(K_X+D_X) = K_{X''} + D_X''$. Let $F$ be the smallest effective divisor for which $(X'', D_X''+F)$ is a log pair. Since $(X,D_X)$ is a log pair, we conclude that every prime component of $F$ has center of codimension at least $2$ on $X$. Since $G \leqslant \mathrm{Bir}(X,D_X)$ for every $g \in G$, we have that $g^*(K_{X''} + D_X'') = K_{X''} + D_X''$. In particular, we have that $g^*F = F$, for every $g \in G$. Note that $(X'', D_X''+F)$ is a $G$-invariant log pair. We run a $G$-equivariant $(K_{X''} + D_X''+F)$-MMP with scaling of an ample $G$-equivariant divisor. Note that
\[
K_{X''} + D_X'' + F \sim_{\mathbb{Q}} F \geq 0.
\]
By the lemma above, we have that $F \subset \mathrm{Bs}_{-}(K_{X''} + D_X'') = \mathrm{Bs}_{-}(F)$. Hence, after finitely many steps of the $G$-equivariant MMP every component of $F$ is contracted. Thus, there is a $G$-equivariant birational contraction $X'' \to X'$ such that the push-forward $K_{X'} + D_X'$ of $K_{X''} + D_X''$ satisfies that $(X',D_X')$ is a log pair. Hence, $G \leqslant \mathrm{Aut}(X',D_X')$ and $(X',D_X')$ is crepant equivalent to $(X,D_X)$ by construction. \qedhere
\end{proof}

\begin{theorem}
 A Jordan bound for $\mathrm{Bir}(\mathbb{P}^3, \Delta)$ is $144$.
\end{theorem}

\begin{proof}
By the previous theorem, there exists a Calabi-Yau pair $(X, D_X)$ crepant birational to $(\PP^3, \Delta)$ such that \( G \leqslant \mathrm{Aut}(X, D_X) \). We have $\mathcal{D}(X, D_X)\simeq_{\mathrm{PL}} \mathcal{D}(\PP^3, \Delta) \simeq_{\mathrm{PL}} S^2$ \cite[Theorem 13]{Koll_r_2015}.

Let \( K \) be the kernel of the action of \( G \) on \( \mathcal{D}(X, D_X) = S^2 \):
\[
1 \longrightarrow K \longrightarrow G \longrightarrow G/K \longrightarrow 1.
\] By \cite[Proposition 2.5.2]{przy}, $K$ is abelian of rank at most 3. 

By potentially passing to an index 2 subgroup of $G$, we know that \( G/K \) is either icosahedral (\( \cong A_5 \), order 60), octahedral (\( \cong S_4 \), order 24), tetrahedral (\( \cong A_4 \), order 12), dihedral (\( D_n \)), or cyclic (\( C_n \)). If \( G/K \) is icosahedral, octahedral, or tetrahedral, then \( [G:K] \leq 60 \). Assume now that \( G/K \) is dihedral or cyclic. Then there exists a subgroup \( A \leqslant G/K \) of index at most 2 (the rotation subgroup) that has a fixed point \( q \in S^2 \). Let $\sigma$ be the unique cell
of $\mathcal D(X,D_X)$ whose relative interior contains $q$; then $\sigma$ is $A$-invariant.
Note that $\sigma$ has dimension $0$, $1$, or $2$, and it corresponds to a stratum of $(X,D_X)$
given by an intersection of $r=\dim(\sigma)+1\le 3$ irreducible components of $D_X$.
Replacing $A$ by a subgroup of $A$ of index at most $3$, we may assume that $A$ fixes one of these
irreducible components $D_0\subset D_X$ setwise. Let \( G' \) be the preimage of \( A \) in \( G \). Then \( [G : G'] \leq 2\cdot 3  \) and we have:
\[
1 \longrightarrow K \longrightarrow G' \longrightarrow A \longrightarrow 1.
\]
Let \( B_0 = \mathrm{Diff}_{D_0}(D_X - D_0) \) be the different divisor, so that \( (D_0, B_0) \) is a log Calabi-Yau surface. Let \( G'' \) be the image of \( G' \) in \( \mathrm{Aut}(D_0, B_0) \). We note that since $(X,D_X)$ is crepant birational to the toric pair $(\mathbb{P}^3,\Delta)$, 
the toric structure is preserved under restriction to boundary components. 
In dimension two, all toric log Calabi-Yau pairs are crepant birational to 
$(\mathbb{P}^2,\Delta)$. Therefore,
$(D_0,B_0)$ obtained by adjunction is crepant birational to 
$(\mathbb{P}^2,\Delta)$, and the induced action $G''$ of $G'$ is a finite subgroup of $\mathrm{Bir}(\mathbb{P}^2,\Delta)$.

By the corresponding 2-dimensional theorem, there exists a subgroup \( A_0 \leqslant G'' \) with \( [G'' : A_0] \leq 12 \) that acts trivially on the dual complex of $(D_0,B_0)$. Let \( \tilde{A}_0 \) be the preimage of \( A_0 \) in \( G' \), so \( [G' : \tilde{A}_0] \leq 12 \). An element \( g \in \tilde{A}_0 \) acts trivially on the dual complex of \((D_0, B_0 )\), since \( B_0 \) contains the intersections of \( D_0 \) with all other components of \(D_X \), the action of \( g \) is trivial on a neighborhood of the vertex of $D_0$ in \( \mathcal{D}(X, D_X) \). As \( \mathcal{D}(X, D_X) = S^2 \) is connected and the action is simplicial, triviality near \( q \) implies triviality everywhere. Hence, \( \tilde{A}_0 \leqslant K  \leqslant G' \), we have:
\[
[G' : K] \leq [G' : \tilde{A}_0] \leq 12.
\]
Recalling that \( [G : G'] \leq 2\cdot 3 \), we conclude \( [G : K] \leq 72 \) in the dihedral or cyclic case.

Taking the abelian normal subgroup to be $K$ and taking into account the potential index 2 subgroup of $G$, we get a Jordan bound of $144$.
\end{proof}

\section{Weak geometric Jordan bound for $\Bir(\Ptwo)$}

\begin{definition}[{\cite[Section 3]{moraga23},\cite[Definition 1.2.1]{prok17},\cite[Conjecture 8.5]{moraga}}]
A group $\mathcal{G}$ has the \emph{geometric Jordan property} if there exists a bound $J$ such that any finite subgroup $G \leqslant \mathcal{G}$ contains an abelian normal subgroup $A \trianglelefteq G$ with index $[G:A] \leq J$, and $A$ is contained in an algebraic torus of $\mathcal{G}$.

A group $\mathcal{G}$ has the \emph{weak geometric Jordan property} if there exists a bound $J$ such that any finite subgroup $G \leqslant \mathcal{G}$ contains an abelian not necessarily normal subgroup $A \leqslant G$ with index $[G:A] \leq J$, and $A$ is contained in an algebraic torus of $\mathcal{G}$.
\end{definition}

\begin{remark}
Any group, or family of groups, satisfying the weak geometric Jordan property also satisfies the geometric Jordan property, possibly with a larger constant.

Indeed, let $G \leqslant \mathcal{G}$ be a finite subgroup, and suppose that $A \leqslant G$ is an abelian subgroup with $[G:A]\leq J$ contained in an algebraic torus of $\mathcal{G}$. Replacing $A$ with 
\(
\operatorname{Core}_G(A):=\bigcap_{g\in G} gAg^{-1}
\), we get that the index is bounded by $J!$. Similarly, any group satisfying the weak Jordan property also satisfies the Jordan property. One may obtain better bounds than the factorial bound above using \cite[Theorem 1.7]{yasinsky2023jordanconstantcremonagroup}.
\end{remark}

We will establish a weak geometric Jordan bound for $\Bir(\Ptwo)$.

A strategy of embedding a finite abelian group $A'\leqslant \mathrm{Bir}(\PP^2)$ in an algebraic torus is to show that it is inside a connected linear algebraic group, which we may see that it has the property that all finite abelian subgroups are embedded in an algebraic torus up to some index, measured by the torsion index. More explicitly, for a finite automorphism subgroup $G\leqslant \Bir(\Ptwo)$ regularized on a minimal $G$-surface $X$, by the weak Jordan property of $\Bir(\Ptwo)$, we know there exists an abelian subgroup $A\leqslant G$ with $[G:A]\leq 288$ \cite[Proposition 1.2.3]{prok17}. We would like to find an abelian subgroup $A'\leqslant A$ that admits a union of $A'$-orbits of  disjoint $(-1)$-curves, so that we may contract those curves $A'$-equivariantly to obtain another surface $X_0$. We would like $\Aut(X_0)$ to be a connected linear algebraic group. Then $A'$ is contained in an algebraic torus of $\mathrm{Aut}(X_0)$ up to some index. Because $\mathrm{Aut}(X_0) \leqslant \mathrm{Bir}(\mathbb{P}^2)$, we get that the torus is actually contained in $\mathrm{Bir}(\mathbb{P}^2)$. By this strategy, we shall see that the linear algebraic groups that we need to consider are only among $\operatorname{PGL}_3(\mathbb{C})$, $\operatorname{PGL}_2(\mathbb{C})\times\operatorname{PGL}_2(\mathbb{C})$, and $\operatorname{Aut}(F_n) \simeq \left(\mathbb{C}^{n+1} \rtimes \operatorname{GL}(2,\mathbb{C})\right)\Big/ \left\{ \begin{pmatrix}\mu & 0 \\ 0 & \mu\end{pmatrix} \mid \mu^n = 1 \right\}$ where $F_n$ is the Hirzebruch surface with $n\geq 2$ \cite[Exercise 2]{BlancCremonaNotes}. We develop the torsion index statement for these groups as follows:

\begin{definition}[Torsion Index \cite{Grothendieck1958,totaro}]
Let $\mathcal{G}$ be a compact connected Lie group and let $T\leqslant \mathcal{G}$ be a maximal torus. Let $N=\dim_{\mathbb{C}}({\mathcal{G}}/T)$. Every character $\chi:T\to \mathbb{C}^{\times}$ determines a complex line bundle 
$L_{\chi}=\mathcal{G}\times_T \mathbb{C}$ on the flag manifold $\mathcal{G}/T$. Consider the subring of the integral cohomology $H^*(\mathcal{G}/T,\mathbb{Z})$ generated by the Chern classes $c_1(L_\chi)\in H^2(\mathcal{G}/T,\mathbb{Z})$ of these line bundles. The \emph{torsion index} of $\mathcal{G}$ is the smallest positive integer $t(\mathcal{G})$ such that $t(\mathcal{G})$ times the class of a point in $H^{2N}(\mathcal{G}/T,\mathbb{Z})\cong \mathbb{Z}$ lies in this subring.
\end{definition}

\begin{remark}
Every complex linear algebraic group $\mathcal{G}$ admits a maximal compact subgroup $C$, unique up to conjugacy \cite[Chapter XV Theorem 3.1]{hoch}. By convention, we let the torsion index of $\mathcal{G}$ to be $t(\mathcal{G}):=t(C)$. Since the complexification of \(C\) is \(\mathcal{G}_{\mathrm{red}}\), where \(\mathcal{G}_{\mathrm{red}}\) denotes the reductive quotient of \(\mathcal{G}\), we have \(t(\mathcal{G}) = t(\mathcal{G}_{\mathrm{red}})\).
\end{remark}

\begin{theorem}[{\cite[4.8]{reichstein}},{\cite[Theorem 1.2]{totaro}}]\label{pgroup}
Let $\mathcal{G}$ be a compact connected Lie group. Then any abelian $p$-subgroup of $\mathcal{G}$ has a subgroup of index dividing the torsion index $t(\mathcal{G})$ which is contained in a maximal torus of $\mathcal{G}$.
\end{theorem}

This statement extends to arbitrary finite abelian subgroups. It was kindly communicated to us and will appear in forthcoming joint work of Reichstein and Scavia. 

\begin{theorem}[Reichstein--Scavia]\label{abegroup}
Let $\mathcal{G}$ be a complex connected reductive group. Then for any finite abelian subgroup $A$ of $\mathcal{G}$, if every Sylow subgroup of $A$ is contained in a torus, then $A$ is contained in a torus.
\end{theorem}

\begin{corollary}\label{cor}
Let $\mathcal{G}$ be one of the groups $\operatorname{PGL}_3(\mathbb{C})$, $\operatorname{PGL}_2(\mathbb{C})\times\operatorname{PGL}_2(\mathbb{C})$, and $\mathrm{Aut}(F_n)$, $n\geq 1$. For every finite abelian subgroup $A\leqslant \mathcal{G}$ there exists a subgroup of index dividing $t(\mathcal{G})$ contained in a maximal torus of $\mathcal{G}$. Moreover, $t(\operatorname{PGL}_3(\mathbb{C}))=3, t(\operatorname{PGL}_2(\mathbb{C})\times\operatorname{PGL}_2(\mathbb{C})) = 4, t(\mathrm{Aut}(F_n))=1 \text{ if $n$ is odd and } = 2 \text{ if $n$ is even}$.
\end{corollary}

\begin{proof}
 We have that the torsion index satisfies $t(\mathcal{G}\times \mathcal{H})=t(\mathcal{G})t(\mathcal{H})$ and $t(\mathcal{G})=t([\mathcal{G},\mathcal{G}])$; see Totaro \cite{totaro}. One has $t(\operatorname{PGL}_n(\mathbb{C}))=n$, so $t(\operatorname{PGL}_3(\mathbb{C}))=3$ and $t(\operatorname{PGL}_2(\mathbb{C})\times\operatorname{PGL}_2(\mathbb{C}))=4$. For $\mathcal{G}=\mathrm{Aut}(F_n)\simeq\left(\mathbb{C}^{n+1} \rtimes \operatorname{GL}(2,\mathbb{C})\right)\Big/ \left\{ \begin{pmatrix}\mu & 0 \\ 0 & \mu\end{pmatrix} \mid \mu^n = 1 \right\}$, we have $t(\mathcal{G})=t(\mathcal{G}_{\mathrm{red}})=t([\mathcal{G},\mathcal{G}])$, and $[\mathcal{G},\mathcal{G}]\simeq \mathrm{SL}_2(\mathbb C)$ if $n$ is odd and $[\mathcal{G},\mathcal{G}]\simeq \operatorname{PGL}_2(\mathbb{C})$ if $n$ is even. Consequently $t(\mathcal{G})=1$ if $n$ is odd and $t(\mathcal{G})=2$ if $n$ is even.

Let $C$ be a maximal compact subgroup of $\mathcal{G}$. Then $C\leqslant \mathcal{G}_{\mathrm{red}}$ and the complexification of $C$ is $\mathcal{G}_{\mathrm{red}}$. The group $A$ is conjugate to a finite subgroup $A_0\leqslant C$ \cite[Chapter XV Theorem 3.1]{hoch}. By \ref{pgroup}, each Sylow subgroup $A_{0}\{p\}$ contains a subgroup $A_{0}\{p\}'$ of index dividing $t(\mathcal{G})_p$ contained in a real torus of $C$, where $t(\mathcal{G})_p$ is the $p$-primary part of $t(\mathcal{G})$. Hence $A_{0}\{p\}'$ is contained in a complex torus of $\mathcal{G}_{\mathrm{red}}$. By Theorem \ref{abegroup}, the product of the $A_{0}\{p\}'$ has index dividing $t(\mathcal{G})$ and is contained in a torus of $\mathcal{G}_{\mathrm{red}}$. Since the product group is finite, the group is contained in a real torus of $C$. After complexification and conjugation we conclude that $A$ is contained in a torus of $\mathcal{G}$.
\end{proof}

\subsection{Del Pezzo surfaces of degree $d \leq 6$}\mbox{}\\
Following Dolgachev and Iskovskikh \cite[Section 6.1]{dolgachev2009finitesubgroupsplanecremona}, there is a group homomorphism from $\mathrm{Aut}(X)$ to $\mathrm{Weyl}(R)$ where $R$ is of type $E_N$ ($N := 9-d= 6, 7, 8$), $D_5$ ($N = 5$), $A_4$ ($N = 4$), or $A_2 + A_1$ ($N = 3$). The image of $\mathrm{Aut}(X)$ in $\mathrm{Weyl}(R)$ may be identified as the group acting on the $(-1)$-curves. Moreover, this homomorphism is injective for $d \leq 5$ as the kernel $K$ fixes the $N$ exceptional curves of the blow-up $X\to \PP^2$. Contracting the exceptional curves $K$-equivariantly, one gets that $K$ is an automorphism subgroup of $\mathbb P^2$ fixing the $N = 9-d$ points of the images of the exceptionals. If $d\leq 5$, then $K$ is trivial.

For $d \leq 5$, we have an injection $\mathrm{Aut}(X) \hookrightarrow \mathrm{Weyl}(R)$, so for any finite $G\leqslant \mathrm{Aut}(X)$, we may take the abelian subgroup $A$ in the weak geometric Jordan property to be the trivial group, and its index is bounded by $|\mathrm{Aut}(X)|$. For $d=5$, $\mathrm{Aut}(X)=120$; for $d=4$, $\mathrm{Aut}(X)\leq 160$, see \cite[Theorem 1.1]{Hosoh1996} or \cite[Theorem 8.6.6]{Dolgachev2012}; for $d=3$, $\mathrm{Aut}(X)\leq 648$ which is obtained by the Fermat cubic surface, \cite[Theorem 9.5.6]{Dolgachev2012}; for $d=2$, $\mathrm{Aut}(X)\leq 336$ \cite[Section 8.7.3]{Dolgachev2012}; for $d=1$, $\mathrm{Aut}(X)\leq 144$ \cite[Section 8.8.4]{Dolgachev2012}.  

For $d = 6$, let $G\leqslant \Aut(X)$ be a finite subgroup, then we have an exact sequence $1 \to A \to G \to \mathrm{Weyl}(A_2 + A_1)$. Since $A$ fixes all $(-1)$-curves, in particular it fixes the three exceptional curves, so doing $A$-equivariant contraction we get that $A \leqslant \operatorname{PGL}_3(\mathbb{C})$ fixing $3$ points, hence it is abelian by Lemma \ref{abelian} and admits a subgroup with index at most $3$ contained in a torus of $\operatorname{PGL}_3(\mathbb{C}) \leqslant \mathrm{Bir}(\mathbb P^2)$ by Corollary \ref{cor}. The index of $A$ is bounded by $|\mathrm{Weyl}(A_2 + A_1)| = 12$.

So in this case, we get a weak geometric Jordan bound $648$.

\subsection{$\mathbb{P}^2$, $\PP^1\times \PP^1$ and minimal conic bundles $\mathbb{F}_n$ , $n\neq 1$}\mbox{}\\
By Corollary \ref{cor}, the torsion index of $\Aut(\PP^2) = \operatorname{PGL}_3(\mathbb{C})$ is 3. For any finite abelian subgroup $A \leqslant \operatorname{Aut}(\mathbb{P}^1 \times \mathbb{P}^1)=(\operatorname{PGL}_2(\mathbb C) \times \operatorname{PGL}_2(\mathbb C))\rtimes \mathbb{Z}/2$, there exists a subgroup of index at most $2$ in $A$ contained in $\operatorname{PGL}_2(\mathbb{C}) \times \operatorname{PGL}_2(\mathbb{C})$, and a further subgroup of index at most 4 contained in a maximal torus by its torsion index. For $\mathrm{Aut}(F_n),n\geq 2$, the torsion index is at most 2. So by the weak Jordan bound $288$ \cite[Proposition 1.2.3]{prok17}, we get a weak geometric Jordan bound $288\cdot 8$ in this case.

\subsection{Singular conic bundles}\mbox{}\\
Let $X$ be a singular conic bundle, i.e., a blow-up of $F_n$ at a finite set of points, no two lying in a fiber of a ruling. Write \(\mathrm{Aut}(X,\pi)\) for the automorphism group preserving the conic bundle structure. 
For any finite subgroup \(G \leqslant \mathrm{Aut}(X,\pi)\) there exists a short exact sequence
\[
1 \rightarrow G_F \rightarrow G \rightarrow G_B \rightarrow 1,
\]
where \(G_B\) acts on the base $\PP^1$ and \(G_F\) acts on the generic fiber. 

By \cite[Proposition 1.2.3]{prok17}, for any finite $G\leqslant\mathrm{Aut}(X,\pi)$, there exists a finite abelian subgroup $A$ in $G$ with index at most 288. We note that under the decomposition $1 \to A_F \to A \to A_{B} \to 1$, the subgroup $A_B$ is cyclic as it is an abelian subgroup of $\operatorname{PGL}_2(\mathbb{C})$. We would like to find a subgroup $A'$ of $A$ such that a union of $A'$-orbits consists of one component from each singular fiber. Then we can use these disjoint $A'$-orbits of $(-1)$-curves to perform $A'$-equivariant contraction to obtain some $\mathbb{F}_n$. In the case of $\mathbb{F}_1$, we can further contract the unique $(-1)$-curve to $\mathbb{P}^2$.

\begin{lemma}
A finite abelian group $A' \leqslant \Aut(X,\pi)$ does not admit a union of $A'$-orbits containing exactly one component from each singular fiber of the conic bundle $\pi:X\to \PP^1$ if and only if there exists an element $g\in A'$ that exchanges the two components of some singular fiber of $\pi$.
\end{lemma}

\begin{proof}
 If such $g$ exists, then there does not exist such a union of $A'$-orbits. Conversely, assume that no element of $A'$ exchanges the two components of any singular fiber.
We construct a union of $A'$-orbits containing exactly one component from each singular fiber. Pick one component $C_1$ of a singular fiber and consider its $A'$-orbit
$
U_1 := \bigcup_{a\in A'} a(C_1).
$
Since no element of $A'$ exchanges components within the same fiber, $U_1$ contains at most one component from each singular fiber.
If some singular fibers have no components contained in $U_1$, pick one component $C_2$ from such a fiber and let
$
U_2 := \bigcup_{a\in A'} a(C_2).
$
Again, $U_2$ contains at most one component from each fiber, and by construction $U_1$ and $U_2$ have no common fiber.
Repeating this procedure for the remaining singular fibers yields finitely many disjoint $A'$-orbits
$
U_1,\ldots,U_r$
whose union
$U := U_1 \cup \cdots \cup U_r$
contains exactly one component from each singular fiber.
Hence, such a union of $A'$-orbits exists.
\end{proof}
We have the following classification of finite abelian groups of $\Bir(\Ptwo)$.
\begin{theorem}[{\cite[Theorem 6]{blanc2006finiteabeliansubgroupscremona}}]\label{thm:finite-abelian-cremona}
The isomorphism classes of finite abelian subgroups of the Cremona group are the following:
\[
\mathbb{Z}/m\mathbb{Z}\times \mathbb{Z}/n\mathbb{Z}\ \text{ for any } m,n\ge 1;\qquad
\mathbb{Z}/2n\mathbb{Z}\times (\mathbb{Z}/2\mathbb{Z})^2\ \text{ for any } n\ge 1;
\]
\[
(\mathbb{Z}/4\mathbb{Z})^2\times \mathbb{Z}/2\mathbb{Z};\qquad
(\mathbb{Z}/3\mathbb{Z})^3;\qquad
(\mathbb{Z}/2\mathbb{Z})^4.
\]
\end{theorem}

\begin{lemma}
For any finite subgroup $G\leqslant \Bir(\Ptwo)$ regularized on a $G$-minimal conic bundle $X$, there exists a subgroup $A'\leqslant G$ such that it admits a union of $A'$-orbits containing exactly one component from each singular fiber of the conic bundle and $[G:A'] \leq 288\cdot 16$.
\end{lemma}

\begin{proof}
By \cite[Proposition 1.2.3]{prok17}, there exists a finite abelian subgroup $A\leqslant G$ with index bounded by $288$ and we know that $A_B$ is cyclic. Consider the conic bundle $X\to \PP^1$. We may assume that $A_B$ fixes $0$ and $\infty\in\PP^1$. If there are singular fibers at $0$ or $\infty\in \PP^1$, then $A$ fixes the singular fibers at $0$ and $\infty$. Consider the group homomorphism $A \to (\mathbb{Z}/2\mathbb{Z})^i$, $i=1$ or $2$ defined by the actions on the $i$ number of singular fibers supported on $0$ and $\infty$, i.e., an automorphism is mapped to $0$ at the $i$-th coordinate if it fixes the two components of the $i$-th singular fiber, and to $1$ if it swaps the two components. The kernel is a subgroup $A_0$ whose elements fix the components of the singular fibers at $0$ and $\infty$. We can perform $A_0$-equivariant blow-down to get another conic bundle with no singular fibers at $0$ and $\infty$. Let $m$ be the number of singular fibers after blow-down.
Consider the exact sequence $$0\to A_{0,F}\to A_0\to A_{0,B}\to 0,$$ consider the map $A_{0,F}\to (\mathbb Z/2\mathbb Z)^m$ given by the actions on the $m$ singular fibers. Let $S$ be the kernel of this map and let $a\in A_0$ such that $\langle a\rangle \cong A_{0,B}$ through the projection. We let $A'= \langle S, a\rangle$. Then $[A:A']=[A:A_0][A_0:A'] \leq [A:A_0][A_{0,F}:S]\leq 2^i 2^{r-i}  = 2^r \leq 16$ where $r$ is the 2-rank of $A$ which is bounded by 4 by the above classification of finite abelian groups. Moreover, we see that $A'$ does not contain any $x$ that swaps the two components of a singular fiber. If not then consider $x=sg$ where $s\in S$ and $g\in\langle a \rangle$ such that $x$ swaps the two components of a singular fiber. Then since $s$ fixes all components and $x$ swaps some singular fiber, then it follows that $g$ swaps the components of the singular fiber so in particular it fixes the base point of the singular fiber, which is not $0$ or $\infty$, so the image of $g$ in the base $\PP^1$ fixes three points and therefore it is trivial in $A_{0,B}$, and so $g\in A_{0,F}$ so $g = 1$ as $g\in \langle a \rangle \cong A_{0,B}$, which shows that $x=s$ does not swap any two components of a singular fiber which is a contradiction. So $A'$ admits a union of $A'$-orbits containing exactly one component from each singular fiber of the conic bundle.
\end{proof}

\begin{proposition}
$\mathrm{Bir}(\mathbb P^2)$ satisfies the geometric Jordan property. A weak geometric Jordan bound is $2^{11} \cdot 3^2$. 
\end{proposition}

\begin{proof}
For a finite subgroup $G$ regularized on a conic bundle, by the previous lemma, we see that there exists an abelian $A'\leqslant G$ that admits such orbits and $[G:A']\leq 288 \cdot 16$. Using this union of orbits, we perform $A'$-equivariant blow-down to obtain some $F_n$ $n\geq 0$, in the case of $F_1$, we may further blow down to obtain $\Ptwo$. By Corollary \ref{cor}, we may find a further subgroup of index at most $4$ contained in a torus. 

For del Pezzo surfaces of $d \leq 6$, $\mathbb P^2$, $\PP^1 \times \PP^1$, and conic bundles $F_n, n\geq 2$, we may find a weak geometric Jordan bound of $288 \cdot 8$ as discussed above. So a weak geometric Jordan bound is given by $288\cdot 16 \cdot 4=2^{11} \cdot 3^2$.
 \end{proof}

\bibliographystyle{plain}
\bibliography{reference}

\begin{thebibliography}{10}

\bibitem{BandmanZarhin2015JordanSurfaces}
Tatiana Bandman and Yuri~G. Zarhin.
\newblock Jordan groups and geometric properties of manifolds.
\newblock {\em Arnold Math. J.}, 10(4):621--635, 2024.

\bibitem{bauer2020mathfraks5equivariantsyzygiesdel}
Ingrid Bauer and Fabrizio Catanese.
\newblock {$\germ {S}_5$}-equivariant syzygies for the del {P}ezzo surface of degree 5.
\newblock {\em Rend. Circ. Mat. Palermo (2)}, 70(1):97--120, 2021.

\bibitem{blanc2006finiteabeliansubgroupscremona}
J\'er\'emy Blanc.
\newblock Finite abelian subgroups of the {C}remona group of the plane.
\newblock {\em C. R. Math. Acad. Sci. Paris}, 344(1):21--26, 2007.

\bibitem{BlancCremonaNotes}
J{\'e}r{\'e}my Blanc.
\newblock Finite subgroups of the {C}remona group of the plane.
\newblock Lecture notes for the EAGER School in Algebraic Geometry, Łukęcin, 2012, 2012.
\newblock Available at \url{https://www.mimuw.edu.pl/~jarekw/EAGER/pdf/FiniteSubgroupsCremona.pdf}.

\bibitem{magma}
Wieb Bosma, John Cannon, and Catherine Playoust.
\newblock The {M}agma algebra system. {I}. {T}he user language.
\newblock {\em J. Symbolic Comput.}, 24(3-4):235--265, 1997.
\newblock Computational algebra and number theory (London, 1993).

\bibitem{BraunFilipazziMoragaSvaldi2022}
Lukas Braun, Stefano Filipazzi, Joaqu\'in Moraga, and Roberto Svaldi.
\newblock The {J}ordan property for local fundamental groups.
\newblock {\em Geom. Topol.}, 26(1):283--319, 2022.

\bibitem{Collins2007}
Michael~J. Collins.
\newblock On {J}ordan's theorem for complex linear groups.
\newblock {\em J. Group Theory}, 10(4):411--423, 2007.

\bibitem{corti}
Alessio Corti and Anne-Sophie Kaloghiros.
\newblock The {S}arkisov program for {M}ori fibred {C}alabi-{Y}au pairs.
\newblock {\em Algebr. Geom.}, 3(3):370--384, 2016.

\bibitem{deFernexEin2002}
Tommaso de~Fernex and Lawrence Ein.
\newblock Resolution of indeterminacy of pairs.
\newblock In {\em Algebraic geometry}, pages 165--177. de Gruyter, Berlin, 2002.

\bibitem{tommaso}
Tommaso de~Fernex, J\'anos Koll\'ar, and Chenyang Xu.
\newblock The dual complex of singularities.
\newblock In {\em Higher dimensional algebraic geometry---in honour of {P}rofessor {Y}ujiro {K}awamata's sixtieth birthday}, volume~74 of {\em Adv. Stud. Pure Math.}, pages 103--129. Math. Soc. Japan, Tokyo, 2017.

\bibitem{Dolgachev2012}
Igor~V. Dolgachev.
\newblock {\em Classical algebraic geometry. A modern view}.
\newblock Cambridge University Press, Cambridge, 2012.

\bibitem{dolgachev2009finitesubgroupsplanecremona}
Igor~V. Dolgachev and Vasily~A. Iskovskikh.
\newblock Finite subgroups of the plane {C}remona group.
\newblock In {\em Algebra, arithmetic, and geometry: in honor of {Y}u. {I}. {M}anin. {V}ol. {I}}, volume 269 of {\em Progr. Math.}, pages 443--548. Birkh\"auser Boston, Boston, MA, 2009.

\bibitem{ducat2022quarticsurfacesvolumepreserving}
Tom Ducat.
\newblock Quartic surfaces up to volume preserving equivalence.
\newblock {\em Selecta Math. (N.S.)}, 30(1):Paper No. 2, 27, 2024.

\bibitem{Grothendieck1958}
Alexander Grothendieck.
\newblock Torsion homologique et sections rationnelles, 1958.
\newblock Exposé 5, Séminaire C. Chevalley 1958, Anneaux de Chow et applications.

\bibitem{hoch}
G.~Hochschild.
\newblock {\em The structure of {L}ie groups}.
\newblock Holden-Day, Inc., San Francisco-London-Amsterdam, 1965.

\bibitem{Hosoh1996}
Toshio Hosoh.
\newblock Automorphism groups of quartic del {P}ezzo surfaces.
\newblock {\em J. Algebra}, 185(2):374--389, 1996.

\bibitem{Jordan1878}
M.~Camille Jordan.
\newblock M\'emoire sur les \'equations diff\'erentielles lin\'eaires \`a{} int\'egrale alg\'ebrique.
\newblock {\em J. Reine Angew. Math.}, 84:89--215, 1878.

\bibitem{Koll_r_2015}
J\'anos Koll\'ar and Chenyang Xu.
\newblock The dual complex of {C}alabi-{Y}au pairs.
\newblock {\em Invent. Math.}, 205(3):527--557, 2016.

\bibitem{przy}
K.~V. Loginov, V.~V. Przyjalkowski, and A.~S. Trepalin.
\newblock {$G$}-coregularity of del {P}ezzo surfaces.
\newblock {\em Tr. Mat. Inst. Steklova}, 329:132--164, 2025.

\bibitem{loginov2024birationalinvariantsvolumepreserving}
Konstantin Loginov and Zhijia Zhang.
\newblock Birational invariants of volume preserving maps.
\newblock \url{arXiv:2410.05036}, 2024.
\newblock To appear in \emph{Algebra \& Number Theory}.

\bibitem{luo2025jordanpropertyautomorphismgroups}
Yujie Luo, Sheng Meng, and De-Qi Zhang.
\newblock Jordan property for automorphism groups of compact varieties.
\newblock \url{arXiv:2502.16956}, 2025.
\newblock To appear in \emph{EM Surv. Math. Sci.}

\bibitem{MengPerroniZhang2018JordanAut}
Sai-Kee Meng, Fabio Perroni, and De-Qi Zhang.
\newblock Jordan property for groups of birational selfmaps.
\newblock {\em Michigan Math. J.}, 67(1):143--166, 2018.

\bibitem{MengZhang2018JordanNonlinear}
Sheng Meng and De-Qi Zhang.
\newblock Jordan property for non-linear algebraic groups and projective varieties.
\newblock {\em Amer. J. Math.}, 140(4):1133--1145, 2018.

\bibitem{moraga23}
Joaqu\'in Moraga.
\newblock Small quotient minimal log discrepancies.
\newblock {\em Michigan Math. J.}, 73(3):593--619, 2023.

\bibitem{mor}
Joaqu\'in Moraga.
\newblock Coregularity of {F}ano varieties.
\newblock {\em Geom. Dedicata}, 218(2):Paper No. 40, 55, 2024.

\bibitem{moraga}
Joaquín Moraga.
\newblock Cluster type varieties.
\newblock \url{arXiv:2602.23584}, 2026.

\bibitem{Popov2014JordanGroups}
Vladimir~L. Popov.
\newblock Jordan groups and automorphism groups of algebraic varieties.
\newblock In {\em Automorphisms in birational and affine geometry}, volume~79 of {\em Springer Proc. Math. Stat.}, pages 185--213. Springer, Cham, 2014.

\bibitem{prokhorov2014jordanpropertycremonagroups}
Yuri Prokhorov and Constantin Shramov.
\newblock Jordan property for {C}remona groups.
\newblock {\em Amer. J. Math.}, 138(2):403--418, 2016.

\bibitem{prok17}
Yuri Prokhorov and Constantin Shramov.
\newblock Jordan constant for {C}remona group of rank 3.
\newblock {\em Mosc. Math. J.}, 17(3):457--509, 2017.

\bibitem{reichstein}
Zinovy Reichstein and Boris Youssin.
\newblock Splitting fields of {$G$}-varieties.
\newblock {\em Pacific J. Math.}, 200(1):207--249, 2001.

\bibitem{serre2009minkowskistyleboundorderfinite}
Jean-Pierre Serre.
\newblock A {M}inkowski-style bound for the orders of the finite subgroups of the {C}remona group of rank 2 over an arbitrary field.
\newblock {\em Mosc. Math. J.}, 9(1):193--208, 2009.

\bibitem{shokurov}
V.~V. Shokurov.
\newblock Complements on surfaces.
\newblock volume 102, pages 3876--3932. 2000.
\newblock Algebraic geometry, 10.

\bibitem{totaro}
Burt Totaro.
\newblock The torsion index of the spin groups.
\newblock {\em Duke Math. J.}, 129(2):249--290, 2005.

\bibitem{yasinsky2023jordanconstantcremonagroup}
Egor Yasinsky.
\newblock The {J}ordan constant for {C}remona group of rank 2.
\newblock {\em Bull. Korean Math. Soc.}, 54(5):1859--1871, 2017.

\end{thebibliography}

\end{document}